\documentclass[a4paper,oneside]{amsart}
\pdfoutput=1

\usepackage{hyperref}
\usepackage{geometry}

\usepackage[utf8]{inputenc}
\usepackage[T1]{fontenc}
\usepackage{lmodern}

\usepackage{amsmath,amssymb,amsthm}
\usepackage[dvipsnames]{xcolor}

\usepackage{bbold}
\usepackage{bm}

\newcommand{\bb}[1]{\boldsymbol{#1}}

\definecolor{cb-yellow}{RGB}{221,170,51}
\definecolor{cb-red} {RGB}{187,85,102}
\definecolor{cb-teal}{RGB}{0,153,136}
\definecolor{cb-blue} {RGB}{0,68,136}
\definecolor{cb-green}{RGB}{17,119,51}
\definecolor{cb-purple} {RGB}{170,68,153}
\definecolor{cb-palegrey} {RGB}{221,221,221}

\usepackage{tikz,pgfplots}
\usetikzlibrary{positioning}
\pgfplotsset{compat=1.18}
\usepackage[backend=biber,style=alphabetic,maxnames=99]{biblatex}
\AtBeginBibliography{\small}  % make bibliography small

\hypersetup{
  pdfauthor={Michael Borinsky, Chiara Meroni, Maximilian Wiesmann},
  pdftitle={Asymptotic number of edge-colored regular graphs},
  pdfsubject={},
  pdfkeywords={},
    colorlinks=true,
    linkcolor=cb-red,
    filecolor=magenta,      
    urlcolor=cb-purple,
    citecolor=cb-yellow,
}

\newtheorem{theorem}{Theorem}
\numberwithin{theorem}{section}

\newtheorem{proposition}[theorem]{Proposition}
\newtheorem{lemma}[theorem]{Lemma}

\newtheorem{definition}[theorem]{Definition}

\theoremstyle{remark}
\newtheorem{remark}[theorem]{Remark}
\newtheorem{examplex}[theorem]{Example}
\newenvironment{example}
{\pushQED{\qed}\examplex}
{\popQED\endexamplex}

\newcommand{\GG}{\mathcal{G}}

\DeclareMathOperator{\argmin}{argmin}

\newcommand{\R}{\mathbb{R}}

\newcommand{\C}{\mathbb{C}}
\newcommand{\Z}{\mathbb{Z}}

\DeclareMathOperator{\crit}{crit}
\DeclareMathOperator{\Aut}{Aut}

\renewcommand{\d}{\mathrm{d}}

\newcommand{\cE}{{P}}

\DeclareMathOperator{\Sym}{\mathbb{S}}
\DeclareMathOperator{\Hess}{Hess}

\newcommand\restr[2]{{
  \left.\kern-\nulldelimiterspace 
  #1
  \vphantom{\big|} 
  \right|_{#2}
  }}

\title{Asymptotic number of edge-colored regular graphs}
\date{}
\author{Michael Borinsky \and Chiara Meroni \and Maximilian Wiesmann}

\begin{document}

\begin{abstract}
  We prove a formula for the asymptotic number of edge-colored regular graphs with a prescribed set of allowed vertex-incidence structures. The formula depends on specific critical points of a polynomial encoding the vertex-incidences. As an application, we compute the expected number of proper $c$-edge-colorings of a large random $k$-regular graph.
\end{abstract}

\maketitle

\section{Introduction}

In this paper we study the asymptotic number of edge-colored regular graphs with a prescribed set of possible vertex-incidence structures. The asymptotic enumeration of regular graphs is a topic with a long history, see for instance~\cite{bender1978asymptotic,bollobas1982asymptotic,mckay1991asymptotic}, and more recently~\cite{liebenau2024asymptotic}.
Since Lov{\'a}sz and co-authors introduced graphons~\cite{lovasz2012large}, large deviation principles have given rise to asymptotic enumeration results, see, e.g.,~\cite[\S 6]{ChatterjeeLargeDeviations}. This framework also naturally allows for edge-colors~\cite[\S 23.2]{lovasz2012large}.
Our setup differs from most of the previously mentioned works in at least three respects. First, we impose precise constraints on the allowed vertex-incidence structures. In particular, this enables us to derive asymptotic counts for proper edge-colorings in Section~\ref{sec:average}. 
Second, we allow graphs to have self-loops and multiple edges. Sometimes such graphs are also called multigraphs in the literature.
Third, we weight graphs by the inverse cardinality of their automorphism group, which is equivalent to counting half-edge labeled graphs. Such weights arise naturally, for instance, in the context of Poly{\'a} enumeration theory~\cite{polya1937kombinatorische}, the theory of combinatorial species~\cite{MR633783}, 
Feynman diagram expansions in quantum field theory~\cite{MR603127}, geometric group theory~\cite{MR830040}, and enumerative geometry~\cite{MR1171758}. Further, such weights appear as compensation factors in classical random graphs studies (see, e.g.,~\cite{MR1220220}).

Beyond these structural differences, the novelty of our approach lies in the application of analytic combinatorial techniques to regular graph counting. The analytic combinatorics framework, popularized by Flajolet and Sedgewick~\cite{MR2483235}, has been remarkably successful in providing precise asymptotic estimates for various discrete structures, most notably in the analysis of large-scale data structures and the distribution of patterns in random mappings. In this paper, we demonstrate that the asymptotic enumeration of edge-colored regular graphs is equivalent to the analytic problem of locating the critical points of a specific polynomial. 

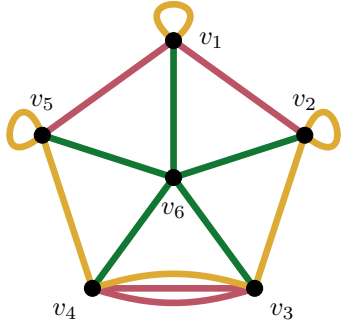
\begin{figure}[ht]
    \centering
    \begin{tikzpicture}[
  x=1ex,y=1ex,
  baseline={([yshift=-1.5ex]current bounding box.center)}
]

% --- geometry ---
\pgfmathsetmacro{\R}{12} % radius of the pentagon (in ex units)

\coordinate (O)  at (0,0);

% outer vertices of a regular pentagon (one at the top)
\coordinate (v1) at ({\R*cos(90)},{\R*sin(90)});   % top
\coordinate (v2) at ({\R*cos(18)},{\R*sin(18)});   % upper right
\coordinate (v3) at ({\R*cos(306)},{\R*sin(306)}); % lower right
\coordinate (v4) at ({\R*cos(234)},{\R*sin(234)}); % lower left
\coordinate (v5) at ({\R*cos(162)},{\R*sin(162)}); % upper left

% --- edges: pentagon cycle ---
\draw[cb-red,  line width=2.5] (v1) -- (v2);
\draw[cb-yellow, line width=2.5] (v2) -- (v3);
\draw[cb-red,  line width=2.5] (v3) -- (v4);
\draw[cb-yellow, line width=2.5] (v4) -- (v5);
\draw[cb-red,  line width=2.5] (v5) -- (v1);

% --- edges: rays from center ---
\draw[cb-green, line width=2.5] (O) -- (v1);
\draw[cb-green, line width=2.5] (O) -- (v2);
\draw[cb-green, line width=2.5] (O) -- (v3);
\draw[cb-green, line width=2.5] (O) -- (v4);
\draw[cb-green, line width=2.5] (O) -- (v5);

% --- self-loops on the three top vertices (each loop contributes degree 2) ---
\path[-, cb-yellow, line width=2.5, distance=1cm] (v1) edge[in=40,out=140,above] (v1);
\path[-, cb-yellow, line width=2.5, distance=1cm] (v2) edge[in=-40,out=70,above] (v2);
\path[-, cb-yellow, line width=2.5, distance=1cm] (v5) edge[in=220,out=110,above] (v5);

% --- extra parallel edges on the bottom edge to make v3 and v4 5-regular ---
% (v3--v4 already exists in the pentagon; add two more parallel edges)
\draw[cb-red,  line width=2.5, bend left=18]  (v3) to (v4);
\draw[cb-yellow, line width=2.5, bend right=18] (v3) to (v4);

% --- vertices (filled dots) ---
\filldraw (O)  circle (3pt);
\filldraw (v1) circle (3pt);
\filldraw (v2) circle (3pt);
\filldraw (v3) circle (3pt);
\filldraw (v4) circle (3pt);
\filldraw (v5) circle (3pt);

% --- labels ---
\node[below=6pt] at (O) {$v_6$};
\node[right=6pt] at (v1) {$v_1$};
\node[above=6pt] at (v2) {$v_2$};
\node[below right=3pt] at (v3) {$v_3$};
\node[below left=3pt] at (v4) {$v_4$};
\node[above=6pt] at (v5) {$v_5$};

\end{tikzpicture}
    \caption{An edge-colored $5$-regular graph, element of $\mathcal{G}_5^3(6)$.}
    \label{fig:graph_intro}
\end{figure}
Let $\mathcal G_k^c(n)$ be the set of isomorphism classes of $k$-regular graphs with $n$ vertices, whose edges can have $c$ different colors. We allow graphs to have multiple edges and self-loops. Each vertex $v$ of such a graph has a multidegree $\deg(v) = (w_1,\ldots,w_c) \in \Z_{\geq 0}^c$, where $w_i$ denotes the number of edges of color $i$ incident to $v$. A self-loop contributes two incidences at a single vertex. The regularity assumption implies that $\sum_{i=1}^c w_i = k$. Let $\mathcal C_k^c$ denote the set of integer compositions of $k$ into $c$ (possibly zero) parts. Each  element of $\mathcal C_k^c$ can occur as the multidegree of a vertex. 
Let $\{\Lambda_{\bb w}\}_{\bb w \in \mathcal C_k^c}$ be a family of parameters indexed by these compositions $\bb w = (w_1,\ldots,w_c) \in \mathcal C_k^c$.
The $\Lambda_{\bb{w}}$ serve as vertex markings.
Figure~\ref{fig:graph_intro} shows a $(k\!=\!5)$-regular graph with six vertices and $c=3$ edge colors: green, red, yellow. With respect to this ordering, the vertex multidegrees are  $(1,2,2)$ for $v_1, v_3, v_4$, $(1,1,3)$ for $v_2, v_5$, and $(5,0,0)$ for $v_6$.\par

The main result of this article, Theorem~\ref{thm:main}, is an asymptotic formula for the weighted sum, 
\begin{align}
\label{eq:AnDef}
A(n) =
\sum_{G \in \mathcal G_k^c(n)}
\frac{1}
{|\Aut(G)|}\prod_{v\in V_G} \Lambda_{\deg(v)}\, ,
\end{align}
where $\Aut(G)$ is the group of automorphisms of a labeled graph $\Gamma$ representing $G$ as defined below. 
Throughout, we use the multi-index notation: bold symbols typically denote $c$-tuples, as in $\bb{w} = (w_1,\dots,w_c)$. The multi-index factorial is $\bb{w}! = w_1!\cdots w_c!$, while $\bb{x}^{\bb{w}}$ is shorthand notation for $x_1^{w_1} \cdots x_c^{w_c}$. Moreover, we use $|\bb{w}| = w_1+\dots +w_c$. We define a polynomial $g\in \R[x_1,\ldots,x_c]$ storing the parameters $\Lambda_{\bb w}$ by
\begin{equation}\label{eq:poly_g}
g(x_1,\ldots,x_c) = 
-\sum_{i=1}^c \frac{x_i^2}{2} +
\sum_{\bb w\in\mathcal C_k^c
}
\Lambda_{\bb w}
\frac{\bb x^{\bb w}}{\bb w!}\, .
\end{equation}
Let $\crit^*(g)$ be the set of critical points of $g$ lying in the complex cone over the sphere $S^{c-1}$,
\[
\crit^*(g) = \left\{ \bb{z} \in \C^c \,:\, \frac{\partial g}{\partial x_1}(\bb{z}) =\ldots = \frac{\partial g}{\partial x_c}(\bb{z}) = 0 \text{ and } \bb{z} = \tau \cdot \bb x \text{ for } \tau \in \C\setminus\{0\} \text{ and } \bb x \in S^{c-1} \right\} .
\]
Among these, let 
\[
  \Psi = \argmin_{\bb z= \tau \cdot \bb x \in \crit^*(g)} |\tau|
\] 
be the subset of critical points of minimal distance to the origin.
Moreover, let $\Hess g$ be the Hessian of $g$ which is the symmetric matrix formed by its second partial derivatives.

\begin{theorem}
  \label{thm:main}
  If all critical points $\bb{z}\in \Psi$ are non-degenerate, i.e.\ $\det (\Hess g(\bb{z}))\neq 0$, then
  \begin{equation}
    \label{eq:main_asymptotic}
    A(n) \sim \frac{1}{2\pi} (\ell-1)! \sum_{\bb{z}\in\Psi} \frac{(-g(\bb{z}))^{-\ell}}{\sqrt{(-1)^{c-1}\det (\Hess g(\bb{z}))}} 
\quad \text{ for large integers } \ell = n\left(\frac{k}{2}-1\right).
  \end{equation}
\end{theorem}
This result generalizes \cite[Theorem 5.3]{BMW24:Bicolored}, which treats the bicolored case $c=2$. The key ingredient in the proof of Theorem~\ref{thm:main} is an expression of $A(n)$ as a multivariate exponential integral (Lemma~\ref{lem:An_exp_integral}) which is then amenable to a saddle point analysis in the spirit of analytic combinatorics (Proposition~\ref{prop:asymptotics_critical_V}).

Let 
${E}^{c}_k(n)$ be the average number of proper $c$-edge-colorings of all $k$-regular vertex-labeled graphs with multiple edges but no self-loops and $n$ vertices, or equivalently the expected number of proper $c$-edge-colorings of random $k$-regular $n$-vertex labeled multigraphs without self-loops. Here, random means with respect to the \emph{uniform} distribution on the set of such vertex-labeled graphs.
As an example  application of Theorem~\ref{thm:main}, we combine it with a classical result on the asymptotic number of regular (multi)graphs~\cite{bender1978asymptotic} to obtain an expression for ${E}^{c}_k(n)$ when $n$ is large while $k$ and $c$ remain fixed.

\begin{theorem}
  \label{thm:multigraphs}
  Let $k\geq 3$ and assume $\ell = n\left(\tfrac{k}{2}-1\right) \in \Z_{>0}$. Then, for large $n\in \Z_{>0}$, we have
  \[
    {E}^{c}_k(n) \sim \left\{ \begin{array}{ll}
      2^{(k-1)/2} \left(\frac{k!}{\sqrt{k}^k}\right)^{n} \exp\left( -\frac{1}{4}(k^2 -4k +3) \right)
& \text{ if } c=k, \text{ and } n \text{ even},\\
        \left( \frac{k-1}{c-1} + 1 \right)^{(1-c)/2} \left( k! \binom{c}{k} \right)^n c^{-\frac{nk}{2}}  
\exp\left( -\frac{1}{4}(k^2 -4k +3) \right)
& \text{ if } c > k,\\
      0 & \text{ else.}
    \end{array}\right.
    \]
\end{theorem}

\medskip
\noindent\textbf{Outline.}
The paper is structured as follows. In Section~\ref{sec:edge-colored_graphs}, we give a precise definition of the graphs considered and derive a counting formula with respect to the vertex multidegrees in Proposition~\ref{prop:gf_graphs}. Section~\ref{sec:asymptotic} is devoted to the proof of the asymptotic behavior of the (weighted) number of such edge-colored graphs, culminating in the proof of Theorem~\ref{thm:main}. Finally, Section~\ref{sec:average} discusses an application of our theorem to count the average
number of proper edge-colorings and proves Theorem~\ref{thm:multigraphs}.

\section*{Acknowledgments}
We thank Steve Melczer for valuable explanations of analytic combinatorics techniques.
CM is supported by Dr. Max R\"ossler, the Walter Haefner Foundation, and the ETH Z\"urich Foundation.
Research at Perimeter Institute is supported in part by the Government of Canada through the Department of Innovation, Science and Economic Development and by the Province of Ontario through the Ministry of Colleges and Universities. 

\section{Edge-colored half-edge labeled graphs}\label{sec:edge-colored_graphs}

In this section, we explain the framework of half-edge labeled graphs. For the case of bicolored graphs, this has appeared in~\cite[\S 3]{BMW24:Bicolored}, the multicolored case has been described in~\cite[\S 2]{LeeYang}.  We restate it here to make the paper self-contained.
Given a finite set $H$ of labels, an $H$-\emph{half-edge labeled graph} is a pair $(V,E)$ where $V$ and $E$ are set-partitions of $H$ with the constraint that $E$ has only blocks of size two. Interpreting each block of $V$ as a vertex gives rise to a graph that can have multiple edges and self-loops. 

\begin{remark}
In classical graph theory, a (simple) graph is defined as a set of vertices $V$ and a set of two-element subsets $E \subset V$, which inherently makes it a \emph{simplicial complex}. However, applications in topology and physics often require a more flexible definition that allows the graph to be a general finite, at most one-dimensional
CW complex (see, e.g.,~\cite{MR603127,MR4896115}). Half-edge labeled graphs are natural combinatorial realizations of these more general complexes.
\end{remark}

Explicitly, in the edge-colored setting, we encode our half-edge labeled graphs as follows.

\begin{definition}
  \label{def:multicolored_half_edge_label}
  Let $H_1,\dots,H_c$ be disjoint sets of half-edge labels, one for each color. An edge-colored $(H_1,\dots,H_c)$-\emph{half-edge labeled graph} $\Gamma$ is a tuple $\Gamma = (V, E_{1},\dots, E_{c})$ such that 
  \begin{enumerate}
    \item $V$, the set of vertices, is a set partition of $H_1 \sqcup \dots \sqcup H_c$;
    \item for $i\in\{1,\dots,c\}$, $E_{i}$ is a set partition of $H_i$ into blocks of size two.
  \end{enumerate}
\end{definition}

The notion of multidegree carries over to edge-colored half-edge labeled graphs: 
for $\Gamma = (V, E_{1},\dots, E_{c})$, $\deg(v) = (w_1,\dots,w_c)$ if the vertex $v \in V$ is incident to exactly $w_i$ half-edges of color $i$, or equivalently $|v \cap H_i| = w_i$. 

An isomorphism between an edge-colored $(H_1,\dots,H_c)$- and an edge-colored $(H_1',\dots,H_c')$-half-edge labeled graph $\Gamma$ and $\Gamma'$  is a tuple of bijections $H_i \rightarrow H_i'$, preserving the partitions $\Gamma = (V, E_{1},\dots, E_{c})$ and $\Gamma' = (V', E'_{1},\dots, E'_{c})$. 
The automorphism group $\Aut(\Gamma)$ consists of all isomorphisms of a graph to itself. We denote the isomorphism class $[\Gamma]$ by $G$ and write $\Aut(G) := \Aut(\Gamma)$ for any half-edge labeled representative.

From now on, we also sometimes drop the adjective edge-colored.
Let $\{\lambda_{\bb w}\}_{\bb w \in \Z_{\geq 0}^c}$ be an infinite set of formal parameters. We use the notation $[\bb x^{\bb a}]f(\bb{x})$ to denote the coefficient of the monomial $\bb x^{\bb a}$ in the formal power series expansion of $f$.
The following proposition is an immediate generalization of Wick's theorem (see, e.g.,~\cite{MR603127}).

\begin{proposition}
  \label{prop:gf_graphs}
The following counting formula holds:
  \begin{equation}\label{eq:gf_graphs}
    \sum_{G \in \GG} \frac{\prod_{v\in V_G} \lambda_{\deg(v)}}{|\Aut(G)|} =  
\sum_{\bb{s} \in \Z_{\geq 0}^c} 
\left(\prod_{i=1}^c (2s_i-1)!! \right)
\cdot [\bb{x}^{2\bb{s}}]\exp\left( \sum_{\bb{w}\in \Z^c_{\geq 0},\, |\bb{w}| \geq 1} \lambda_{\bb{w}} \frac{\bb{x}^{\bb{w}}}{\bb{w}!} \right),
  \end{equation}
where $\GG$ is the set of isomorphism classes of edge-colored half-edge labeled graphs. 
\end{proposition}

\begin{proof}
We interpret the right-hand side as a weighted enumeration of edge-colored graphs using the half-edge definition. Expansion of the exponential on the right-hand side and extracting the coefficient of $\bb x^{2\bb s}=x_1^{2s_1}\cdots x_c^{2s_c}$ gives
\begin{equation}\label{eq:exponential_and_block_partitions}
    [x_1^{2s_1}\cdots x_c^{2s_c}]
    \exp\left( \sum_{\bb{w}\in \Z^c_{\geq 0},\, |\bb{w}| \geq 1} \lambda_{\bb{w}} \frac{\bb{x}^{\bb{w}}}{\bb{w}!} \right)
    = 
    \sum_{\{t_{\bb{u}}\}} 
    \frac{1}{\prod_{\bb{u}} t_{\bb{u}}!\, \bb{u}!^{t_{\bb{u}}}} \prod_{\substack{\bb{u}\in \Z^c_{\geq 0}}} \lambda_{\bb{u}}^{t_{\bb{u}}},
\end{equation}
where the last sum ranges over all functions
$\bb u\mapsto t_{\bb u}\in\Z_{\ge0}$ satisfying
$\sum_{\bb u}t_{\bb u}u_i=2s_i$ for each $i=1,\dots,c$.
This condition implies that only finitely many $t_{\bb u}$ are nonzero.

Fix a tuple of sets $(H_1,\ldots,H_c)$ with $|H_i|=2s_i$ and write $H:=H_1\sqcup\cdots\sqcup H_c$.
A choice of multiplicities $\{t_{\bb u}\}$ determines a class of partitions of $H$ into blocks such that, for each $\bb u=(u_1,\dots,u_c)$, exactly $t_{\bb u}$ blocks
contain $u_i$ elements from $H_i$.
Each block represents a vertex with multidegree $\bb u$.
The number of such partitions is
\begin{equation}\label{eq:block_partitions}
\frac{(2\bb s)!}{\prod_{\bb u}t_{\bb u}!\,\bb u!^{t_{\bb u}}},
\end{equation}
which follows from the orbit--stabilizer theorem applied to the following natural action. The $c$-fold product of symmetric groups $\Sym_{2s_1} \times \dots \times \Sym_{2s_c}$ acts  on the set of partitions of $H_1 \sqcup \dots \sqcup H_c$. A partition with $t_{\bb{u}}$ blocks with $u_i$ elements from $H_i$ is stabilized by a subgroup isomorphic to $(\Sym_{u_1} \times \dots \times \Sym_{u_c})^{t_{\bb{u}}} \rtimes \Sym_{t_{\bb{u}}}$, permuting the elements inside each block and the blocks themselves.
Thus the coefficient \eqref{eq:exponential_and_block_partitions} enumerates vertex sets with
prescribed multidegrees $\bb u$ such that $\sum_{\bb u}t_{\bb u}u_i=2s_i$, weighted by the factor
$\prod_v\lambda_{\deg(v)}$.

For each color $i$, an edge of color $i$ is obtained by pairing two half-edges in
$H_i$.
The number of perfect matchings on $H_i$ is $(2s_i-1)!!=(2s_i-1) \cdot (2s_i-3) \cdots 3 \cdot 1$, and these choices are
independent for different colors.
Choosing such matchings produces a collection of colored edges; self-loops and
multiple edges are allowed, since paired half-edges may lie in the same block or in
distinct blocks.
The total number of ways to pair all half-edges into edges is therefore
\[
\prod_{i=1}^c(2s_i-1)!! \, .
\]

Combining a partition of $H$ with matchings on each $H_i$ yields a half-edge labeled
edge-colored graph which represents an isomorphism class of a graph $G$ 
with
exactly $s_i$ edges of color $i$.
For such a fixed isomorphism class, the group
$\Sym_{2s_1}\times\cdots\times\Sym_{2s_c}$
acts freely on the set of half-edge labelings by permuting half-edges within
each color.
The stabilizer of $G$ is precisely the automorphism group $\Aut(G)$.
By the orbit--stabilizer theorem, the number of half-edge labeled graphs representing $G$ is
\[
\frac{(2\bb s)!}{|\Aut(G)|}.
\]
Multiplying \eqref{eq:exponential_and_block_partitions} by the factor
\[
\frac{1}{(2\bb s)!}\prod_{i=1}^c(2s_i-1)!!
\]
has the effect of pairing half-edges into edges and then forgetting the labeling of
half-edges.
Consequently, each isomorphism class $G$ appears with weight
$\prod_{v\in V_G}\lambda_{\deg(v)}/|\Aut(G)|$.
Summing over all $\bb s\in\Z_{\ge0}^c$ proves the stated identity.
\end{proof}

\begin{example}
We illustrate Proposition~\ref{prop:gf_graphs} in the case of $c=3$ colors and graphs with exactly one edge.
An edge-colored graph with exactly one edge is of the form \scalebox{.35}{\begin{tikzpicture}[x=1ex,y=1ex,baseline={([yshift=-1.5ex]current bounding box.center)}]
    \coordinate (v);
    \coordinate [above=4 of v] (vm1);
    \coordinate [right=3 of vm1] (v11);
    \draw[black, line width = 3.2] (v11) circle(3);
    \filldraw (vm1) circle (3.8pt);
\end{tikzpicture}%
} or \scalebox{.35}{\begin{tikzpicture}[x=1ex,y=1ex,baseline={([yshift=-1.5ex]current bounding box.center)}]
    \coordinate (vm);
    \coordinate [left=4 of vm] (v0);
    \coordinate [right=4 of vm] (v1);
    \draw[black, line width = 3.2] (v0) to (v1);
    \filldraw (v0) circle (3.8pt);
    \filldraw (v1) circle (3.8pt);
\end{tikzpicture}%
} and it must have its edge colored by one of the three colors.
For a fixed color $i\in\{1,2,3\}$, the loop graph has $|\Aut(G)|=2$, and contributes 
$\lambda_{2\bb e_i}/2$,
where $\bb e_i$ is the vector with a $1$ in coordinate $i$ and zeros otherwise.
The segment graph also has $|\Aut(G)|=2$, and contributes
$\lambda_{\bb e_i}^2/2$.
Summing over the three possible colors, the contribution of all one-edge graphs to the left-hand side of~\eqref{eq:gf_graphs} equals
\[
\sum_{i=1}^3\left(
\frac{\lambda_{\bb e_i}^2}{2}
+
\frac{\lambda_{2\bb e_i}}{2}
\right).
\]
On the right-hand side of~\eqref{eq:gf_graphs}, graphs with exactly one
edge correspond to summands with
\[
\bb s\in\{(1,0,0),(0,1,0),(0,0,1)\},
\]
since $|2 \bb s| = 2$.
For $\bb s=\bb e_i$, we have
$\prod_{j=1}^3(2s_j-1)!! = (2\cdot1-1)!!=1$,
and we must compute the coefficient of $x_i^2$ in the expansion of the exponential function.
Only the terms $\lambda_{\bb e_i}x_i$ and $\lambda_{2\bb e_i}x_i^2/2$ contribute to
this coefficient, and expanding yields
\[
[x_i^2]\exp\!\left(
\sum_{\substack{\bb w\in\Z_{\ge0}^3\\|\bb w|\ge1}}
\lambda_{\bb w}\frac{\bb x^{\bb w}}{\bb w!}
\right)
=
\frac{\lambda_{\bb e_i}^2}{2}
+
\frac{\lambda_{2\bb e_i}}{2}.
\]
Summing over $i=1,2,3$ gives the contribution of one-edge graphs on the left-hand side.
\end{example}

\section{Asymptotics for regular graphs}\label{sec:asymptotic}

The goal of this section is to prove the main result, Theorem~\ref{thm:main}.
Recall that $\mathcal C_k^c$ is the set of all integer compositions of $k$ into $c$ (possibly zero) blocks, namely all $\bb w\in \Z_{\geq 0}^c$ such that $|\bb w| = k$. 
Consider the homogeneous polynomial $V \in \R[\bb x]$ of degree $k\geq 3$,
\begin{align*}
V(\bb x) = 
\sum_{ \bb w \in\mathcal C_k^c
}
\Lambda_{\bb w}
\frac{\bb x^{\bb w}}{\bb w!}\, ,
\end{align*}
which relates to the polynomial defined in~\eqref{eq:poly_g} as $g(\bb x) = -\tfrac{1}{2}\sum_i x_i^2 + V(\bb x)$. In contrast to the formal variables $\lambda_{\bb w}$, we think of the $\Lambda_{\bb w}$ as fixed real parameters.
Via Proposition~\ref{prop:gf_graphs}, the polynomial $V$ encodes a vertex-incidence structure for $k$-regular graphs: when it is used as the argument of the exponential function on the right-hand side of \eqref{eq:gf_graphs}, then the only relevant summands on the left-hand side are $k$-regular graphs.\par

Our proof strategy for Theorem~\ref{thm:main} consists of using the homogeneity of $V$ to find an exponential integral representation over the sphere of the numbers
\begin{align}
\label{eq:An}
A(n) =
\sum_{G \in \mathcal G_k^c(n)}
\frac{\prod_{v\in V_G} \Lambda_{\deg(v)}}
{|\Aut(G)|}\, ,
\end{align}
where $n$ is the number of vertices of $G$ (Lemma~\ref{lem:An_exp_integral}). To such an integral representation we can apply the saddle point method to obtain an asymptotic formula in terms of critical points of $V$ (Proposition~\ref{prop:asymptotics_critical_V}). 
Translating this formula in terms of critical points of the polynomial $g$ yields a computationally more tractable expression, which is our main result, Theorem~\ref{thm:main}. We start by providing an integral formula for $A(n)$. This result also appears as Proposition 2.6 in~\cite{LeeYang}.

\begin{lemma}
  \label{lem:An_exp_integral}
  Let $V\in \R[\bb{x}]$ be a homogeneous polynomial of degree $k$ and let us denote $m = \tfrac{n k}{2}$. Then, for any integer $n\geq 0$ such that $m\in\Z$, we have
  \begin{equation}
    \label{eq:An_exp_integral}
    A(n) = \frac{2^{m + \frac{c-2}{2}} \Gamma\left(m +\tfrac{c}{2}\right)}{(2\pi)^{\frac{c}{2}} n!} \int_{S^{c-1}} V(\bb{x})^{n} \,\omega_{S^{c-1}},
  \end{equation}
  where $\omega_{S^{c-1}}$ denotes the standard volume form of the sphere $S^{c-1}$, and $\Gamma$ is the gamma function.
\end{lemma}

\begin{remark}
The statement implies that, for a fixed polynomial $V$, the sequence $A(0), A(1), \ldots,$ is \emph{holonomic}~\cite{MR4621625}. Without invoking the holonomicity, a rather straightforward generalization of~\cite[\S 4]{BMW24:Bicolored} provides a (moderately efficient) algorithm to compute the values of $A(n)$. An algorithm that produces explicit $p$-recursions for sequences $A(n)$ would be useful. Techniques developed in~\cite{ChyMis25:DifferentialEquationsGeneratingFunction,BCL25:MultivariateIntegrationDModules} might provide such an algorithm.
\end{remark}

\begin{proof}
  Since the graphs $G$ contributing to the sum~\eqref{eq:An} for $A(n)$ are $k$-regular with $n$ vertices, the number of edges of $G$ is $m=\tfrac{n k}{2}$ which must be an integer. From Proposition~\ref{prop:gf_graphs}, we have 
  \begin{equation}
    \label{eq:An_series}
    A(n) = \sum_{\substack{\bb{s} \in \Z_{\geq 0} \\ |\bb{s}| = m}} \left(\prod_{i=1}^c (2s_i-1)!! \right)
\cdot [\bb{x}^{2\bb{s}}]\exp\left( \sum_{\bb w \in\mathcal C_k^c
} \Lambda_{\bb{w}} \frac{\bb{x}^{\bb{w}}}{\bb{w}!} \right).
  \end{equation}
  The $\bb{x}^{2\bb{s}}$-coefficient of the exponential on the right-hand side has degree $\tfrac{2m}{k} = n$ in the coefficients $\Lambda_{\bb{w}}$. Using the homogeneity of $V$, this implies
  \[
    [\bb{x}^{2\bb{s}}]\exp\left( \sum_{\bb w \in\mathcal C_k^c} \Lambda_{\bb{w}} \frac{\bb{x}^{\bb{w}}}{\bb{w}!} \right) = [\bb{x}^{2\bb{s}}] \frac{V(\bb{x})^{n}}{n!}.
  \]
  Recall the Gaussian integral identity
  \[
    \frac{1}{\sqrt{2\pi}} \int_{-\infty}^{\infty} e^{-\frac{x^2}{2}} x^t \d x = \left\{ \begin{array}{ll}
      (t-1)!! & \text{if } t \text{ is even,} \\
      0 & \text{if } t \text{ is odd.}
    \end{array} \right.
  \]
  This, used on each of the $c$ double factorials, enables us to rewrite \eqref{eq:An_series} as 
  \[
    A(n) = \frac{1}{(2\pi)^{c/2} n!} \int_{\R^c} \exp\left( -\frac{1}{2} \sum_{i=1}^{c} x_i^2 \right) V(\bb{x})^{n} \d\bb{x}.
  \]
  Finally, the homogeneity of $V$ allows us to switch to polar coordinates rescaling $V(r\cdot \bb{x}) = r^k V(\bb{x})$, so we can rewrite the integral above as
  \begin{align*}
    A(n) & = \frac{1}{(2\pi)^{c/2} n!} \int_{0}^{\infty} r^{c-1} \left( \int_{\partial B_r(0)} \exp\left( -\frac{1}{2} \sum_{i=1}^{c} x_i^2 \right) V(\bb{x})^{n} \,\omega_{\partial B_r(0)} \right) \d r\\
    & = \frac{1}{(2\pi)^{c/2} n!} \int_{0}^{\infty} e^{-\frac{r^2}{2}} r^{c-1 + kn} \d r \int_{S^{c-1}} V(\bb{x})^{n} \,\omega_{S^{c-1}},
  \end{align*}
  where $B_r(0)$ is the Euclidean ball of radius $r$ centered at $0$, and $\omega_{\partial B_r(0)}$ or $\omega_{S^{c-1}}$ denote the standard volume form on the sphere.
  Together with the identity
  \[
    \int_{0}^{\infty} e^{-\frac{r^2}{2}} r^{nk+c-1} \d r = \int_{0}^{\infty} e^{-q} (2q)^{\frac{nk + c-2}{2}} \d q = 2^{m + \frac{c-2}{2}} \Gamma\left(m + \tfrac{c}{2}\right)\,,
  \]
  this yields the stated expression.
\end{proof}

Lemma~\ref{lem:An_exp_integral} can be used to deduce the asymptotics of $A(n)$ by performing a saddle point analysis of the integral in \eqref{eq:An_exp_integral}. 
We use the asymptotic equivalence notation
$A(n) \sim B(n)$  to denote $\lim_{n\rightarrow \infty} A(n)/B(n) = 1$.

\begin{proposition}
  \label{prop:asymptotics_critical_V}
  Let $\Phi$ be the set of global maxima of the function $S^{c-1}\subset \R^c \rightarrow \R,~ \bb{x}\mapsto |V(\bb{x})|$, 
  and assume that all of these are non-degenerate, i.e.\ the Hessian matrix has nonzero determinant. If $m = \frac{nk}{2}\in\Z$, then, for large $n$, we have
  \begin{equation}
    \label{eq:asymptotics_sphere}
    A(n) \sim \frac{k^{m + \frac{c-1}{2}} \left(\frac{k}{2}-1\right)^{n-m}\sqrt{\frac{k}{2}-1} }{\sqrt{8}\pi}\, \Gamma\!\left(m-n \right) \sum_{\bb{x}\in\Phi} \frac{V(\bb{x})^{n}}{\sqrt{(-1)^{c-1} \det(\Hess_{S^{c-1}} f(\bb x))}}\, ,
  \end{equation}
  where $\Hess_{S^{c-1}}$ is the negative definite Hessian taken intrinsically on the sphere $S^{c-1}$, in local coordinates $\varphi_{\bb x}(\bb y)$ around $\bb x$, with $\bb y \in \R^{c-1}$ and $f_{\bb x}(\bb y) = \log \frac{V(\varphi_{\bb x}(\bb{y}))}{V(\bb{x})}$. 
\end{proposition}

\begin{remark}
    Consider a function $f:S^{c-1}\subset \R^c\to \R$ and $\bb x \in S^{c-1}$ a non-degenerate maximizer for $f$. Let $\varphi_{\bb x}:\R^{c-1}\to S^{c-1}$ be a system of local coordinates, namely a smooth diffeomorphism from a neighborhood of the origin $\bb 0$ to a neighborhood of $\bb x$, so that $\varphi_{\bb x}(\bb 0) = \bb x$. Let us denote by $\bb y$ the coordinates in $\R^{c-1}$. Then, the spherical Hessian $\Hess_{S^{c-1}}f (\bb{x})$ of $f$ at $\bb{x}$ (in local coordinates $\bb{y}$) is the $(c-1)\times (c-1)$ negative definite matrix with entries $b_{ij}$ such that 
    \[
    f(\varphi_{\bb x}(\bb{y})) = \frac{1}{2} \sum_{i,j=1}^{c-1} b_{ij}\, y_i  y_j + o(y_1^2 + \dots + y_{c-1}^2),
    \]
    for $\bb y$ in a neighborhood of $\bb 0$.
\end{remark}

\begin{proof}
  From Lemma~\ref{lem:An_exp_integral} we see that for large $n$, since $S^{c-1}$ is compact, the main contribution to the integral~\eqref{eq:An_exp_integral} comes from points near the maxima of $\bb{x}\mapsto |V(\bb{x})|$ on $S^{c-1}$. The statement follows from a saddle point expansion around these maxima, see~\cite[\S 4.6]{deBruijn2014asymptotic}, and applications of Stirling's formula:
  \begin{align*}
    A(n) & \sim \frac{2^{m + \frac{c-2}{2}} \left(m +\frac{c-2}{2}\right)!}{\sqrt{2\pi} \, n! \, n^{\frac{c-1}{2}}} \sum_{\bb{x}\in\Phi} \frac{V(\bb{x})^{n}}{\sqrt{\det(-\Hess_{S^{c-1}} f(\bb{x}))}} \\
    & \sim \frac{1}{2\sqrt{\pi}}
    \, k^{\frac{c-1}{2}}
    \, n^{-\frac{1}{2}}
    k^{\frac{nk}{2}}
    \, e^{n(1-\frac{k}{2})}
    \, n^{n(\frac{k}{2}-1)} \sum_{\bb{x}\in\Phi} \frac{V(\bb{x})^{n}}{\sqrt{(-1)^{c-1} \det(\Hess_{S^{c-1}} f(\bb{x}))}}.
  \end{align*}
  We will see later that it will be convenient to use the Gamma function of $n(\frac{k}{2}-1)=m-n$ to rewrite this expression.  
This is done using again Stirling's formula, to get
  \[
  A(n) \sim \frac{k^{m + \frac{c-1}{2}} \left(\frac{k}{2}-1\right)^{n-m}\sqrt{\frac{k}{2}-1}}{\sqrt{8}\pi}\, \Gamma\!\left(m-n \right) \sum_{\bb{x}\in\Phi} \frac{V(\bb{x})^{n}}{\sqrt{(-1)^{c-1} \det(\Hess_{S^{c-1}} f(\bb x))}}.
  \]
The difference $n-m$, which appears prominently here, is the \emph{Euler characteristic} of the graph.
\end{proof}

  \begin{remark}
    In \cite{LeeYang} it is shown that if one introduces a complex parameter into the polynomial $V(\bb{x})$, $A(n)$ can be written, for large $n$, as a sum over critical points of $V$ restricted to $\mathcal{S} = \{\bb{x}\in\C^c \,:\, x_1^2+\dots+x_c^2 = 1\}$. The space $\mathcal{S}$ can be thought of as a complex compactification of the real sphere $S^{c-1}$. 
    The sum representation comes from a decomposition of the integral in Lemma~\ref{lem:An_exp_integral} into integrals over Lefschetz thimbles associated to each critical point. Therefore, Proposition~\ref{prop:asymptotics_critical_V} can be seen as a specialization of \cite[Corollary 3.4]{LeeYang} to the case where $V(\bb{x})$ is defined over $\R$, so that not all Lefschetz thimbles contribute to the integral but only those corresponding to global maxima or minima of $V$ on $\mathcal{S}_{\R} = S^{c-1}$.
  \end{remark}

The expression \eqref{eq:asymptotics_sphere} is not ideal in that one needs to compute the Hessian matrices in local coordinates. We remedy this issue by providing a different asymptotic expression in terms of the Hessian of $g$ on $\R^c$, which is straightforward to compute. The following lemmata prepare this.

Recall that we denote by $\crit^*(g)$ the critical points of $g$ in the complex cone over $S^{c-1}$,
\[
\crit^*(g) = \left\{ \bb{z} \in \C^c \,:\, \frac{\partial g}{\partial x_1}(\bb{z}) =\ldots = \frac{\partial g}{\partial x_c}(\bb{z}) = 0 \text{ and } \bb{z} = \tau \cdot \bb x \text{ for } \tau \in \C\setminus\{0\} \text{ and } \bb x \in S^{c-1} \right\} ,
\]
and 
\[
  \Psi = \argmin_{\bb z= \tau \cdot \bb x \in \crit^*(g)} |\tau|.
\] 

\begin{lemma}
  \label{lem:critical_point_identification}
  Consider the polynomial $g$ as defined in \eqref{eq:poly_g} and let $\Psi$ be the set of critical points among $\crit^*(g)$ with minimal distance to the origin as defined above. Let $\sim$ be the equivalence relation on $S^{c-1}$ identifying antipodal points. Then, the map 
  \[
    \nu \,:\, \Psi \rightarrow \Phi / \sim, \quad \bb{z} = \tau\cdot \bb{x} \mapsto \bb{x}
  \]
  is onto and every fiber has cardinality $k-2$.
\end{lemma}

\begin{proof}
    The map $\nu$ is well-defined only onto the quotient space since $\tau \bb x = (-\tau) (-\bb x)$, hence we want to identify $\bb x$ with $-\bb x$.
  Let $\tau\bb{x}\in\crit^*(g)$ be a critical point. It satisfies
  \begin{equation}
    \label{eq:critical_eqs}
    \tau x_i = \frac{\partial V}{\partial x_i}(\tau \bb{x}) \text{ for } i=1,\dots,c.
  \end{equation}
  Multiplying these by $\tau x_i$, summing over all $i$ and using homogeneity, we obtain
  \begin{equation}
    \label{eq:tau_def}
    \tau^2 = \tau^2 \sum_{i=1}^{c} x_i^2 = \sum_{i=1}^c \tau x_i \frac{\partial V}{\partial x_i}(\tau \bb{x}) = kV(\tau \bb{x}) = k \tau^k V(\bb{x}),
  \end{equation}
  so $\tau^{2-k} = kV(\bb{x})$. Moreover, $\bb{x}$ is a critical point of $\restr{V}{S^{c-1}}$. Indeed, from \eqref{eq:critical_eqs} we see that $\bb{x}$ satisfies the Lagrange equations on the sphere
  \[
    \frac{\partial V}{\partial x_i}(\bb{x}) - 2\lambda x_i = 0 \quad \text{for all } i=1,\dots,c
  \]
  with Lagrange multiplier $\lambda = \tfrac{\tau^{2-k}}{2}$.
  Since $k\geq 3$, any $\tau$ satisfying \eqref{eq:tau_def} with minimal $|\tau|$ gives a maximum of $|V(\bb{x})|$. This implies $\bb{x}\in\Phi$. As there are $k-2$ complex solutions to $\tau^{2-k} = kV(\bb{x})$ for a fixed $\bb{x}$, this proves the claim. 
\end{proof}

\begin{lemma}
  \label{lem:Hessian_transformation}
  Let $\bb{z} = \tau\bb{x}\in\crit^*(g)$ and let $f_{\bb x}(\bb y) = \log \frac{V(\varphi_{\bb x}(\bb{y}))}{V(\bb{x})}$, in local coordinates $\bb y$ on the sphere given by $\varphi_{\bb {x}}:\R^{c-1} \to S^{c-1}$, with $\varphi_{\bb {x}}(\bb 0) = \bb {x}$. Then,
  \[
  \det (\Hess_{S^{c-1}} f(\bb x)) = \frac{k^{c-1}}{k-2} \det (\Hess g ( \bb{z} )).
  \]
\end{lemma}

\begin{proof}
    Using standard techniques from differential geometry, e.g.,~ Corollary 5.16 in \cite{boumal2023introduction}, we can rewrite the Hessian on the sphere as
    \[
    \Hess_{S^{c-1}} f(\bb x) = \frac{1}{V(\bb{x})} \left( P_{\bb{x}}  \mathrm{Hess}\, V(\bb{x}) P_{\bb{x}}^T - \langle \nabla V(\bb{x})\, ,\, \bb{x} \rangle \cdot \mathrm{Id}_{c-1}\right),
    \]
    where $P_{\bb{x}}$ is the $(c-1)\times c$ projection matrix onto the hyperplane orthogonal to $\bb{x}\in S^{c-1}$, namely the tangent space to $S^{c-1}$ at $\bb{x}$, and $\Hess V$ denotes the standard Hessian matrix of $V$ as a function on $\R^c$.
    Because $V$ is $k$-homogeneous, by the Euler relation one deduces that $\frac{1}{V(\bb{x})} \langle \nabla V(\bb{x})\, ,\, \bb{x} \rangle = k$, hence 
    \[
   \Hess_{S^{c-1}} f(\bb x) = \frac{1}{V(\bb{x})} P_{\bb{x}}  \mathrm{Hess} V(\bb{x}) P_{\bb{x}}^T - k\, \mathrm{Id}_{c-1} = \frac{\tau^{2-k}}{V(\bb{x})} P_{\bb{x}} \mathrm{Hess} V(\tau \bb{x}) P_{\bb{x}}^T - k\, \mathrm{Id}_{c-1},
    \]
    where we used the fact that the Hessian matrix of $V$ is homogeneous of degree $k-2$. 
    Changing to the Hessian of $g$, where the entries are by definition $\partial_i\partial_j g$, we get 
    \[
    \Hess_{S^{c-1}} f(\bb x) = \frac{\tau^{2-k}}{V(\bb{x})} P_{\bb{x}} \left( \mathrm{Id}_{c} +  \Hess g(\bb{z}) \right) P_{\bb{x}}^T - k\, \mathrm{Id}_{c-1}.
    \]
    Using the formula for $\tau$ in \eqref{eq:tau_def}, and the property $P_{\bb{x}}P_{\bb{x}}^T = \mathrm{Id}_{c-1}$, this reduces to 
    \[
    \Hess_{S^{c-1}} f(\bb x) = k P_{\bb{x}}  \Hess g(\bb{z}) P_{\bb{x}}^T.
    \]
    We claim that the Hessian matrix of $g$ has an eigenvector in the radial direction, with eigenvalue $\lambda_{\rm rad}$. Indeed, using Euler's relation on $\nabla V$, homogeneous of degree $k-1$, and using \eqref{eq:critical_eqs}, we get
    \[
    \left(\Hess g(\bb{z})\right) \bb{z} = -\bb{z} + \left(\Hess V(\bb{z})\right) \bb{z} = -\bb{z} + (k-1)\nabla V(\bb{z}) = (k-2) \bb{z}.
    \]
    Therefore, $\lambda_{\rm rad} = k-2$, and we denote the remaining eigenvalues $\lambda_1,\ldots,\lambda_{c-1}$. Then, $\prod_{i=1}^{c-1} \lambda_i = \det \left( P_{\bb{x}}  \Hess g(\bb{z}) P_{\bb{x}}^T \right)$ and we obtain
    \[
    \det \left( \Hess_{S^{c-1}} f(\bb x)\right) = \frac{k^{c-1}}{k-2} \det \left(\Hess g(\bb{z})\right) . \qedhere
    \]
\end{proof}

We can now prove the main result for the asymptotic behavior of $A(n)$ in terms of critical points of $g$, as $n\to \infty$.

\begin{proof}[Proof of Theorem~\ref{thm:main}]
  Lemma~\ref{lem:critical_point_identification} allows us to match points in $\Phi$ with points in $\Psi$. Take $\bb{z} = \tau\bb{x} \in \Psi$. Using \eqref{eq:tau_def}, we get
  \begin{align*}
    g(\bb{z}) &= -\frac{\tau^2}{2} + V(\tau\bb{x}) = -\frac{\tau^2}{2} + \tau^k V(\bb{x}) = \tau^2 \frac{2-k}{2k}\\
    V(\bb{x})^{n} &= \left( \frac{\tau^{2-k}}{k} \right)^{n} = k^{-n}\tau^{2n(1-\frac{k}{2})} = k^{-\frac{nk}{2}}\left(\frac{k}{2}-1\right)^{n(\frac{k}{2}-1)} (-g(\bb{z}))^{-n(\frac{k}{2}-1)}.
  \end{align*}
  Combining the latter with Lemma~\ref{lem:Hessian_transformation} and accounting for the multiplicities from Lemma~\ref{lem:critical_point_identification}, we can rewrite the asymptotic expression for $A(n)$ from Proposition~\ref{prop:asymptotics_critical_V} as 
  \begin{gather*}
     \frac{k^{\frac{nk}{2} + \frac{c-1}{2}} \left(\frac{k}{2}-1\right)^{-n(\frac{k}{2}-1)}\sqrt{\frac{k-2}{2}}}{\sqrt{8}\pi}\, \Gamma\!\left(n\left(\frac{k}{2}-1\right) \right) \frac{2}{k-2} 
    \sum_{\bb{z}\in\Psi} \frac{k^{-\frac{nk}{2}}\left(\frac{k}{2}-1\right)^{n(\frac{k}{2}-1)} (-g(\bb{z}))^{-n(\frac{k}{2}-1)}}{\sqrt{\frac{k^{c-1}}{k-2}} \sqrt{(-1)^{c-1} \det \left( \Hess g(\bb{z}) \right)}} \\
    = \frac{1}{2 \pi} \Gamma\!\left(n\left(\frac{k}{2}-1\right) \right) \sum_{\bb{z}\in\Psi} \frac{(-g(\bb{z}))^{-n(\frac{k}{2}-1)}}{\sqrt{(-1)^{c-1}\det \left( \Hess g(\bb{z}) \right)}}. 
  \end{gather*}
  Substituting $\ell = n(\frac{k}{2}-1) \in \Z$ proves the statement.
\end{proof}

Note that when $c=1$, the statement of Theorem~\ref{thm:main} reduces to \cite[Theorem 3.3.1]{Borinsky:2018mdl}, whereas for $c=2$ we get \cite[Theorem 5.3]{BMW24:Bicolored}. 
An open question regards the generalization of the result to graphs that are not necessarily regular. The difficulty lies in finding an integral representation over a suitable domain with a nice complex compactification, as the sphere in our case. This is not straightforward in the non-regular scenario and we leave it for future research.

\section{Average number of proper edge-colorings}\label{sec:average}

As an application of our main result, Theorem~\ref{thm:main}, we prove an asymptotic formula for the weighted sum of proper edge-colorings with $c$ colors of $k$-regular graphs. 
Combined with a known result on the number of $k$-regular vertex-labeled (multi)graphs without self-loops \cite{bender1978asymptotic}, this gives then rise to an asymptotic formula for the expected number of proper $c$-edge-colorings of a random such graph, which is stated in Theorem \ref{thm:multigraphs}. Here, \emph{random} means with respect to the uniform distribution on the set of $k$-regular vertex-labeled graphs without self-loops.\par

From now on, instead of edge-coloring we simply write coloring. As before, we assume $k\geq 3$. A $c$-coloring of a graph $G$ is a map $E_G \rightarrow \{1,\dots,c\}$, assigning a color to each edge. It is \emph{proper} if no vertex is incident to two edges of the same color. By weighted we mean that we take an $|\Aut(G)|^{-1}$-weighted sum over all $k$-regular graphs $G$, i.e.\ we wish to determine the large-$n$ asymptotics of 
\begin{equation}
  \label{eq:num_colorings}
  \cE^{c}_k(n) := \sum_{\substack{G~ k\text{-regular}\\ |V_G| = n}} \frac{\#\text{(proper $c$-colorings of $G$)}}{|\Aut(G)|}.
\end{equation}
The automorphism group in \eqref{eq:num_colorings} considers $G$ as a monocolored graph.
Clearly, in order for a $c$-coloring to be proper, we need $c\geq k$. In fact, a standard result from graph theory asserts that for any \emph{simple} $k$-regular graph $G$ there exists a $c$-coloring with $c = k+1$. For $c=k$, a $c$-coloring is also called a 1-factorization of $G$. The average number of such 1-factorizations has been computed, e.g.,~in \cite{robinson1994almost}, though there are two important differences to our number $\cE^{k}_k(n)$: the first is regarding the probability distribution on the space of graphs which we do not take to be uniform but weighted by $|\Aut(G)|^{-1}$. The second concerns the type of graphs: \cite{robinson1994almost} consider simple graphs, whereas our graphs are allowed to have multiple edges.

\begin{theorem}
  \label{thm:asymptotic_edge_colorings}
  Let $k\geq 3$ and assume $\ell = n\left(\tfrac{k}{2}-1\right) \in \Z_{>0}$. Then, for large $n\in \Z_{>0}$, we have
  \[
    \cE^{c}_k(n) \sim \left\{ \begin{array}{ll}
      \frac{(\ell-1)! \cdot 2^{k/2}}{2\pi} \left( \frac{2}{k-2} \right)^{\ell-1/2} & \text{ if } c=k, \text{ and } n \text{ even},\\
      \frac{(\ell-1)! \cdot \sqrt{k-2}}{2\pi} \left( \frac{k-1}{c-1} + 1 \right)^{(1-c)/2} \left( k \binom{c}{k} \right)^n c^{-\frac{nk}{2}}
\left( \frac{2k}{k-2} \right)^\ell & \text{ if } c > k,\\
      0 & \text{ else.}
    \end{array}\right.
  \]
\end{theorem}

We prepare the proof of Theorem~\ref{thm:asymptotic_edge_colorings} with two lemmata. A proper $c$-coloring corresponds to an edge-colored graph $G$ arising from the elementary symmetric polynomial
\[
  V_k^c := e_k(x_1,\dots,x_c) = \sum_{1\leq a_1 < a_2 <\dots < a_k \leq c} x_{a_1}x_{a_2}\dots x_{a_k}.
\]
We obtain the following reformulation as an edge-colored graph counting problem.

\begin{lemma}
  \label{lem:coloring_as_colored_graph}
  With vertex weights $\{\Lambda_{\bb{w}}\}_{\bb w \in\mathcal C_k^c}$ from $V_k^c$, we have $\cE^{c}_k(n) = A(n)$.
\end{lemma}

\begin{proof}
  It remains to show the following. Suppose there are exactly $N$ distinct colorings of $G$ such that the induced edge-colored graphs $\Gamma_1,\dots,\Gamma_N$ are isomorphic. Then $|\Aut(\Gamma_i)| = |\Aut(G)|/N$. \par
  Let $G$ have $2w$ many half-edges, while $\Gamma_i$ (for any $i$) has $2w_j$ half-edges of the $j$th color. Note that $w = w_1+\dots + w_c$. The group $\Sym_{2w}$ naturally acts on the half-edge labels of $\Gamma$, but also on the disjoint union $\sqcup_{1\leq j\leq c} H_j$ of half-edge labels of $\Gamma_i$. In particular, we obtain an action of $\Aut(G)$ on the set of edge-colored half-edge labeled graphs, and $\Gamma_1,\dots,\Gamma_N$ form an orbit of this action. Since the stabilizer is $\Aut(\Gamma_i)$, the claim follows from the orbit-stabilizer theorem.
\end{proof}

\begin{example}
  Consider the graph \scalebox{0.5}{\begin{tikzpicture}[x=1ex,y=1ex,baseline={([yshift=-1.5ex]current bounding box.center)}]
    \coordinate (vm);
    \coordinate [left=4 of vm] (v0);
    \coordinate [right=4 of vm] (v1);
    \draw[black, line width = 2.5] (v0) to (v1);
    \draw[black, line width = 2.5] (v0) to[bend left=45] (v1);
    \draw[black, line width = 2.5] (v0) to[bend right=45] (v1);
    \filldraw (v0) circle (2.8pt);
    \filldraw (v1) circle (2.8pt);
\end{tikzpicture}%
}. It admits six distinct 3-colorings:
  \begin{figure}[ht!]
    \centering
    \scalebox{0.7}{\begin{tikzpicture}[x=1ex,y=1ex,baseline={([yshift=-1.5ex]current bounding box.center)}]
    \coordinate (vm);
    \coordinate [left=4 of vm] (v0);
    \coordinate [right=4 of vm] (v1);
    \draw[cb-red, line width = 2.5] (v0) to[bend left=45] (v1);
    \draw[cb-yellow, line width = 2.5] (v0) to (v1);
    \draw[cb-green, line width = 2.5] (v0) to[bend right=45] (v1);
    \filldraw (v0) circle (2.8pt);
    \filldraw (v1) circle (2.8pt);
\end{tikzpicture}%
}\hspace{1em}
    \scalebox{0.7}{\begin{tikzpicture}[x=1ex,y=1ex,baseline={([yshift=-1.5ex]current bounding box.center)}]
    \coordinate (vm);
    \coordinate [left=4 of vm] (v0);
    \coordinate [right=4 of vm] (v1);
    \draw[cb-red, line width = 2.5] (v0) to[bend left=45] (v1);
    \draw[cb-green, line width = 2.5] (v0) to (v1);
    \draw[cb-yellow, line width = 2.5] (v0) to[bend right=45] (v1);
    \filldraw (v0) circle (2.8pt);
    \filldraw (v1) circle (2.8pt);
\end{tikzpicture}%
}\hspace{1em}
    \scalebox{0.7}{\begin{tikzpicture}[x=1ex,y=1ex,baseline={([yshift=-1.5ex]current bounding box.center)}]
    \coordinate (vm);
    \coordinate [left=4 of vm] (v0);
    \coordinate [right=4 of vm] (v1);
    \draw[cb-yellow, line width = 2.5] (v0) to[bend left=45] (v1);
    \draw[cb-red, line width = 2.5] (v0) to (v1);
    \draw[cb-green, line width = 2.5] (v0) to[bend right=45] (v1);
    \filldraw (v0) circle (2.8pt);
    \filldraw (v1) circle (2.8pt);
\end{tikzpicture}%
}\hspace{1em}
    \scalebox{0.7}{\begin{tikzpicture}[x=1ex,y=1ex,baseline={([yshift=-1.5ex]current bounding box.center)}]
    \coordinate (vm);
    \coordinate [left=4 of vm] (v0);
    \coordinate [right=4 of vm] (v1);
    \draw[cb-yellow, line width = 2.5] (v0) to[bend left=45] (v1);
    \draw[cb-green, line width = 2.5] (v0) to (v1);
    \draw[cb-red, line width = 2.5] (v0) to[bend right=45] (v1);
    \filldraw (v0) circle (2.8pt);
    \filldraw (v1) circle (2.8pt);
\end{tikzpicture}%
}\hspace{1em}
    \scalebox{0.7}{\begin{tikzpicture}[x=1ex,y=1ex,baseline={([yshift=-1.5ex]current bounding box.center)}]
    \coordinate (vm);
    \coordinate [left=4 of vm] (v0);
    \coordinate [right=4 of vm] (v1);
    \draw[cb-green, line width = 2.5] (v0) to[bend left=45] (v1);
    \draw[cb-red, line width = 2.5] (v0) to (v1);
    \draw[cb-yellow, line width = 2.5] (v0) to[bend right=45] (v1);
    \filldraw (v0) circle (2.8pt);
    \filldraw (v1) circle (2.8pt);
\end{tikzpicture}%
}\hspace{1em}
    \scalebox{0.7}{\begin{tikzpicture}[x=1ex,y=1ex,baseline={([yshift=-1.5ex]current bounding box.center)}]
    \coordinate (vm);
    \coordinate [left=4 of vm] (v0);
    \coordinate [right=4 of vm] (v1);
    \draw[cb-green, line width = 2.5] (v0) to[bend left=45] (v1);
    \draw[cb-yellow, line width = 2.5] (v0) to (v1);
    \draw[cb-red, line width = 2.5] (v0) to[bend right=45] (v1);
    \filldraw (v0) circle (2.8pt);
    \filldraw (v1) circle (2.8pt);
\end{tikzpicture}%
}
  \end{figure}\\
  These are all isomorphic as edge-colored graphs. However, their automorphism groups are decreased compared to the monocolored graph, namely
  \[
    \Aut(\scalebox{0.5}{\begin{tikzpicture}[x=1ex,y=1ex,baseline={([yshift=-1.5ex]current bounding box.center)}]
    \coordinate (vm);
    \coordinate [left=4 of vm] (v0);
    \coordinate [right=4 of vm] (v1);
    \draw[black, line width = 2.5] (v0) to (v1);
    \draw[black, line width = 2.5] (v0) to[bend left=45] (v1);
    \draw[black, line width = 2.5] (v0) to[bend right=45] (v1);
    \filldraw (v0) circle (2.8pt);
    \filldraw (v1) circle (2.8pt);
\end{tikzpicture}%
}) \cong \Sym_3 \rtimes \Sym_2, \quad \Aut(\scalebox{0.5}{\begin{tikzpicture}[x=1ex,y=1ex,baseline={([yshift=-1.5ex]current bounding box.center)}]
    \coordinate (vm);
    \coordinate [left=4 of vm] (v0);
    \coordinate [right=4 of vm] (v1);
    \draw[cb-red, line width = 2.5] (v0) to[bend left=45] (v1);
    \draw[cb-yellow, line width = 2.5] (v0) to (v1);
    \draw[cb-green, line width = 2.5] (v0) to[bend right=45] (v1);
    \filldraw (v0) circle (2.8pt);
    \filldraw (v1) circle (2.8pt);
\end{tikzpicture}%
}) \cong \Sym_2.
  \]
  The relation $|\Aut(\scalebox{0.5}{\begin{tikzpicture}[x=1ex,y=1ex,baseline={([yshift=-1.5ex]current bounding box.center)}]
    \coordinate (vm);
    \coordinate [left=4 of vm] (v0);
    \coordinate [right=4 of vm] (v1);
    \draw[black, line width = 2.5] (v0) to (v1);
    \draw[black, line width = 2.5] (v0) to[bend left=45] (v1);
    \draw[black, line width = 2.5] (v0) to[bend right=45] (v1);
    \filldraw (v0) circle (2.8pt);
    \filldraw (v1) circle (2.8pt);
\end{tikzpicture}%
})|^{-1}\#\text{(proper edge $c$-colorings of $G$)} = |\Aut(\scalebox{0.5}{\begin{tikzpicture}[x=1ex,y=1ex,baseline={([yshift=-1.5ex]current bounding box.center)}]
    \coordinate (vm);
    \coordinate [left=4 of vm] (v0);
    \coordinate [right=4 of vm] (v1);
    \draw[cb-red, line width = 2.5] (v0) to[bend left=45] (v1);
    \draw[cb-yellow, line width = 2.5] (v0) to (v1);
    \draw[cb-green, line width = 2.5] (v0) to[bend right=45] (v1);
    \filldraw (v0) circle (2.8pt);
    \filldraw (v1) circle (2.8pt);
\end{tikzpicture}%
})|^{-1}$ holds.
\end{example}

By Lemma \ref{lem:coloring_as_colored_graph}, to prove Theorem \ref{thm:asymptotic_edge_colorings} we need to compute the asymptotics of $A(n)$ for $V^{c}_k$. To invoke the asymptotic formula \eqref{eq:main_asymptotic}, it is most convenient to compute the critical point set $\Psi$ by first computing the maxima of $|V^c_k|$ on $S^{c-1}$ and then mapping these to $\Psi$ via Lemma~\ref{lem:critical_point_identification}.

\begin{lemma}
  \label{lem:max_all_coordinates_equal}
  The maximum of $S^{c-1} \rightarrow \R,~ \bb{x} \mapsto |V^{c}_k(\bb{x})|$ is attained at 
  \begin{equation} 
  \label{eq:maximizer_positive}
    \bb{x} = (c^{-1/2},c^{-1/2},\dots,c^{-1/2}).
  \end{equation}
\end{lemma}

\begin{proof}
  Since $S^{c-1}$ is symmetric and $V^c_k$ has only positive coefficients, the maximum is attained on $S_+ := S^{c-1}\cap \R^c_{\geq 0}$ and is a maximum of $V^c_k$. On $S_+$, the polynomial $V^c_k$ is Schur-concave. By an application of the Cauchy--Schwarz inequality to $\bb{x}$ and the all-ones vector, we obtain $\sum_{i=1}^{c} x_i \leq \sqrt{c}$, with equality if and only if all $x_i$ are equal. Then the statement follows from the majorization property of Schur-concave functions.
\end{proof}

\textit{Proof of Theorem \ref{thm:asymptotic_edge_colorings}.} If either $\ell$ or $n$ are not integral, $\cE^{c}_k(n) = 0$ since $n = |V_G|$ and $\ell = \chi(G) = |V_G| - |E_G|$ are combinatorial quantities. Let us now assume that both $\ell, n\in \Z_{>0}$.\par
To obtain all maxima of $|V^{c}_k(\bb{x})|$ we need to insert signs to the coordinates of \eqref{eq:maximizer_positive} such that every monomial in $V^c_k$ is positive, or every monomial is negative. Here, we need to distinguish two cases: if $k=c$, an arbitrary sign configuration is allowed, so there are $2^k$ maxima. If, on the other hand, $c>k$, then there are only two such sign configurations, the all $+$ and the all $-$ configuration. We study these cases separately.

\subsection*{The case $c=k$.}

All maxima of $|V^k_k|$ are $\Phi = \{ ( \pm k^{-1/2}, \pm k^{-1/2}, \dots, \pm k^{-1/2} ) \}$. 
As in the proof of Theorem~\ref{thm:main}, we have $\tau^{2-k} = kV^k_k(\bb{x}) = \sigma k^{(2-k)/2}$, where $\sigma$ is the product of signs in the coordinates of $\bb{x}$. 
Therefore, $\tau = \sqrt{k} \zeta_{\sigma}$, where $\zeta_{\sigma}$ is a $(k-2)^{\text{th}}$ root of $\sigma$. Then, for $\bb{z} = \tau\bb{x}$, 
\begin{equation}
  \label{eq:g_c=k}
  g(\bb{z}) = \tau^2 \frac{2-k}{2k} = \frac{(2-k)\zeta_\sigma^2}{2}.
\end{equation}
If $n = \frac{2\ell}{k-2}$ is odd, then the summands in $\sum_{\bb{z}\in\Psi} (-g(\bb{z}))^{-\ell}$ cancel each other and $A(n) = 0$. This should be expected since any proper edge-coloring with $k$ colors induces a perfect matching which in particular implies that the number of vertices needs to be even. If $n$ is even, $(-g(\bb{z}))^{-\ell}$ does not depend on $\sigma$.

\begin{lemma}
  \label{lem:Hess_det}
  For $c=k$ and any $\bb{z}\in\Psi$, the Hessian determinant is 
  \[\det \left( \Hess g(\bb{z}) \right) = (-1)^{k-1} 2^{k-1}(k-2).\]
\end{lemma}

\begin{proof}
  The Hessian is easily computed as
  \[
    (\Hess g(\bb{z}))_{ij} = \left\{ \begin{array}{ll}
      -1 & \text{if } i=j, \\
      \frac{z_1z_2\dots z_k}{z_i z_j} & \text{else.} 
    \end{array} \right.
  \]
  Here, $z_i = \tau x_i = \sigma_i \zeta_\sigma$, where $\sigma_i$ denotes the sign of the $i^{\rm th}$ coordinate. Therefore, $(\Hess g(\bb{z}))_{ij} = \sigma_i\sigma_j$ for $i\neq j$. After multiplying the $i^{\rm th}$ row by $\sigma_i$ and performing some row operations we find
  \begin{align*}
    \det \left( \Hess g(\bb{z}) \right) & = \sigma \det \begin{pmatrix}
      -\sigma_1 & \sigma_2 & \dots & \sigma_k \\
      \sigma_1 & -\sigma_2 & \dots & \sigma_k \\
      \vdots & \vdots & \ddots & \vdots \\ 
      \sigma_1 & \sigma_2 & \dots & -\sigma_k \\
    \end{pmatrix}
    = (-1)^{k-1} \sigma \det \begin{pmatrix}
      \sigma_1 & \ast & \dots & \ast \\
      0 & 2\sigma_2 & \dots & \ast \\
      \vdots & \vdots & \ddots & \vdots \\ 
      0 & 0 & \dots & 2(k-2)\sigma_k \\
    \end{pmatrix} \\
    & = (-1)^{k-1} 2^{k-1} (k-2) \sigma^2 = (-1)^{k-1} 2^{k-1} (k-2). \qedhere
  \end{align*}
\end{proof}

Plugging the two computations from \eqref{eq:g_c=k} and Lemma \ref{lem:Hess_det} into \eqref{eq:main_asymptotic} we obtain 
\begin{equation}
  \label{eq:num_k_chromatic}
  \cE^{k}_k(n) \sim \frac{1}{2\pi} \Gamma(\ell) 2^k \frac{k-2}{2} \frac{\left( \frac{k-2}{2} \right)^{-\ell}}{2^{(k-1)/2} \sqrt{k-2}} = \frac{\Gamma(\ell)2^{k/2}}{2\pi} \left( \frac{2}{k-2} \right)^{\ell-1/2}.
\end{equation}

\subsection*{The case $c > k$.} 

Here, the maxima of $|V^c_k|$ on $S^{c-1}$ are the two points 
\[\Phi = \left\{\pm (c^{-1/2}, c^{-1/2},\dots, c^{-1/2})\right\}.\] 
As in the previous case, let $\tau$ satisfy $\tau^{2-k} = k V^c_k(\bb{x}) = \sigma k \binom{c}{k} c^{-k/2}$, with $\sigma\in\{\pm 1\}$. The sign $\sigma$ is negative iff $\bb{x}$ has all negative signs and $k$ is odd. The evaluation at $g$ is then given by
\begin{equation}
  \label{eq:g_evaluation}
  g(\bb{z}) = \tau^2 \frac{2-k}{2k} = \left(\sigma k \binom{c}{k}\right)^{2/(2-k)} c^{k/(k-2)} \frac{2-k}{2k}.
\end{equation}
If $k$ is odd and $n = 2\ell/(k-2) \in \Z_{>0}$, then $n$ must be even, so the sign disappears in $(-g(\bb{z}))^{-\ell}$. \par

\begin{lemma}
  \label{lem:Hess_det2}
  Let $c>k$ and denote $\tilde{g} := \frac{k-1}{c-1}$. Then,
  for any $\bb{z}\in\Psi$, the Hessian determinant is 
  \[\det \left( \Hess g(\bb{z}) \right) = (-1)^{c-1}(\tilde{g} +1)^{c-1} (\tilde{g}(c-1) -1).\]
\end{lemma}

\begin{proof}
To compute the Hessian determinant note that 
\[
  (\Hess g(\bb{z}))_{ij} = \left\{ \begin{array}{ll}
    e_{k-2}(z_1,\dots,\hat{z}_i,\dots,\hat{z}_j,\dots,z_c) & \text{ if } i\neq j, \\
    -1 & \text{ if } i=j.
  \end{array} \right.
\]
Since $z_i = \tau x_i = \pm \zeta_\sigma \left( k \binom{c}{k} c^{-k/2} \right)^{1/(2-k)} c^{-1/2}$, we obtain 
\[
  e_{k-2}(z_1,\dots,\hat{z}_i,\dots,\hat{z}_j,\dots,z_c) = (\pm 1)^{k-2} \sigma \cdot \frac{c \binom{c-2}{k-2}}{k \binom{c}{k}} = \frac{k-1}{c-1} = \tilde{g}.
\]
Indeed, the expression above is independent of the signs: if $\bb{x}$ has positive signs, every sign is positive. If $\bb{x}$ has negative signs and $k$ is even, then $\sigma = +1$ and $(-1)^{k-2}\sigma = +1$. Lastly, if $\bb{x}$ has negative signs and $k$ is odd, then $\sigma = -1$, so $(-1)^{k-2}\sigma = +1$. Therefore, we have
\[
  \det \left( \Hess g(\bb{z}) \right) = \det \begin{pmatrix}
      -1 & \tilde{g} & \dots & \tilde{g} \\
      \tilde{g} & -1 & \dots & \tilde{g} \\
      \vdots & \vdots & \ddots & \vdots \\ 
      \tilde{g} & \tilde{g} & \dots & -1 \\
    \end{pmatrix}
    = (-1)^{c-1}(\tilde{g} +1)^{c-1} (\tilde{g}(c-1) -1),
\]
where the last equality follows since $\Hess g(\bb{z}) = (\tilde{g}\mathbb{1}_{c\times c} - (\tilde{g} + 1)\mathrm{Id}_{c})$ has eigenvalues $(\tilde{g}(c-1) -1)$ and $-(\tilde{g} +1)$, the latter with multiplicity $c-1$ (here, $\mathbb{1}_{c\times c}$ is the all-ones matrix of size $c\times c$). 
\end{proof}

Combining Lemma \ref{lem:Hess_det2} with \eqref{eq:g_evaluation}, we get from Theorem \ref{thm:main}
\begin{align*}
    \cE^{c}_k(n) & \sim \frac{\Gamma(\ell)(k-2)}{2\pi} \frac{\left( \left(k\binom{c}{k}\right)^{2/(2-k)} c^{k/(k-2)} \left(\frac{k-2}{2k}\right) \right)^{-\ell}}{\left(\frac{k-1}{c-1} + 1 \right)^{(c-1)/2} \left(\frac{k-1}{c-1} (c-1) -1 \right)^{1/2}} \\
    & = \frac{\Gamma(\ell)\sqrt{k-2}}{2\pi} \left( \frac{k-1}{c-1} + 1 \right)^{(1-c)/2} \left( k \binom{c}{k} \right)^n c^{\frac{k\ell}{(2-k)}} \left( \frac{2k}{k-2} \right)^\ell.
\end{align*}
This concludes the proof of Theorem~\ref{thm:asymptotic_edge_colorings}. \qed

To prove Theorem~\ref{thm:multigraphs}, we combine the count from Theorem~\ref{thm:asymptotic_edge_colorings} with a known result from \cite{bender1978asymptotic} to obtain the average number of proper $c$-colorings of $k$-regular vertex-labeled graphs without self-loops. This number is denoted ${E}^{c}_k(n)$, where, as before, $n$ is the number of vertices. 

\begin{proof}[Proof of Theorem~\ref{thm:multigraphs}]
  Let $G$ be a properly $c$-colored $k$-regular graph. As it is properly colored, such a graph cannot have a self-loop. An automorphism in $\Aut(G)$ cannot permute the multiple edges between vertices, because all these edges are required to have different colors. Therefore, no nontrivial $\alpha \in \Aut(G)$ fixes all vertices of $G$ and the group homomorphism $\Aut(G) \rightarrow \Sym_n$ that maps an automorphism of $G$ to its induced permutation on the vertex set is injective.
Let $\mathcal M_n^{k,c}$ be the set of 
properly $c$-colored $k$-regular graphs without self-loops with a vertex-labeling. An action of 
$\Sym_n$ on $\mathcal M_n^{k,c}$ permutes the vertex labels. By the preceding argument, the stabilizer of a vertex-labeled graph under this action equals the automorphism group as we defined it in the first part of the paper. It follows from the orbit-stabilizer theorem that $n!\, \cE^{c}_k(n)$ counts the number of vertex-labeled properly $c$-colored $k$-regular graphs without self-loops. 

Equivalently, $n!\, \cE^{c}_k(n)$ is the number of tuples $(M^{(1)}, \ldots, M^{(c)}) \in (\{0,1\}^{n\times n})^c$ of $\{0,1\}$-matrices that fulfill the following constraints:
\begin{align*}
M^{(i)}_{\ell,m} &= M^{(i)}_{m,\ell}, &
M^{(i)}_{\ell,\ell} &= 0, &
\sum_{\ell} M^{(i)}_{m,\ell} &\leq 1, & \text{and }&&
\sum_{i,\ell} M^{(i)}_{m,\ell} &= k  \,,
\end{align*}
where $\ell,m$ and $i$ 
range over $1,\ldots,n$ and $1,\ldots,c$ 
respectively. \par

To obtain the average number of proper $c$-colorings of a $k$-regular vertex-labeled graph without self-loops ${E}^{c}_k(n)$, we need to divide $n!\, \cE^{c}_k(n)$ by the number of $k$-regular vertex-labeled graphs without self-loops. This number is known from the literature and can be obtained from specializing the formula in \cite[Theorem 1]{bender1978asymptotic}. We explain how to use this formula below, matching with the notation in \cite{bender1978asymptotic}.\par

Let $M = (M_{ij})_{ij} = \mathbb{1}_{n\times n} - \mathrm{Id}_{n}$ be the $n\times n$ matrix with zeros on the diagonal and ones elsewhere. Let $\bb{r} = k\cdot \mathbb{1}_{n\times 1}$ be the vector with all entries equal to $k$. The set of $k$-regular vertex-labeled graphs without self-loops can be identified with the set
\[
  \GG(M,\bb{r},k) := \left\{ (g_{ij})_{ij} \in \Z^{n\times n} \text{ symmetric, } g_{ij}\in \{0,1,\dots,k\},\, g_{ij} = 0 \text{ if } M_{ij} = 0,\, \sum_j g_{ij} = r_i  \right\}.
\]
Then from \cite[Theorem 1]{bender1978asymptotic} 
we obtain
\[
  |\GG(M,\bb{r},k)| \sim \left( \frac{kn}{e} \right)^{kn/2} \frac{\sqrt{2} \exp\left( \frac{1}{4}(k^2 -4k +3) \right)}{(k!)^n} 
\]
Dividing $n!\, \cE^{c}_k(n)$ by $|\GG(M,\bb{r},k)|$ 
and using Stirling's formula to get 
\[
  \sqrt{2} \left( \frac{nk}{e} \right)^{nk/2} (k!)^{-n}
  \sim
  n! \frac{\sqrt{k-2}}{2\pi} \Gamma(\ell) \left(\frac{2k}{k-2}\right)^{\ell} ((k-1)!)^{-n} 
  \, ,
\]
gives the statement.
\end{proof}

\printbibliography

\vspace{1em}
\noindent \textsc{Michael Borinsky}\\
\textsc{Perimeter Institute}\\
\textsc{31 Caroline St N, Waterloo, N2L 2Y5, Ontario, Canada}\\
\url{mborinsky@perimeterinstitute.ca}\\

\noindent \textsc{Chiara Meroni}\\
\textsc{ETH Institute for Theoretical Studies}\\
\textsc{Scheuchzerstrasse 70, 8006 Zürich, Switzerland}\\
\url{chiara.meroni@eth-its.ethz.ch}\\

\noindent \textsc{Maximilian Wiesmann}\\
\textsc{Center for Systems Biology Dresden \\
Max Planck Institute of Molecular Cell Biology and Genetics\\
Pfotenhauerstrasse 108, 01307 Dresden, Germany}\\
\url{wiesmann@pks.mpg.de}

\end{document}